\documentclass[12pt]{amsart}
\usepackage[utf8]{inputenc}
\usepackage[scale=0.69]{geometry}
\usepackage{amssymb,tikz,graphicx,hyperref,microtype}
\usepackage{xcolor}

\newcommand{\R}{\mathbb R}

\newcommand{\Pp}{\mathcal P}

\newtheorem{theorem}{Theorem}[section]

\newtheorem{lemma}[theorem]{Lemma}

\newtheorem{proposition}{Proposition}[section]

\title{Equicovering masses in the Euclidean plane}

\author[M. A. Espinosa-García]{Manuel A. Espinosa-García}
\address[M. A. Espinosa-García]{Centro de Ciencias Matemáticas, UNAM Campus Morelia, Morelia, Mexico}
\email{esgama@matmor.unam.mx}

\author[L. Martínez-Sandoval]{Leonardo Martínez-Sandoval}
\address[L. Martínez-Sandoval]{Facultad de Ciencias, UNAM, Ciudad de México, México}
\email{leomtz@ciencias.unam.mx}

\author[E. Roldán-Pensado]{Edgardo Roldán-Pensado}
\address[E. Roldán-Pensado]{Centro de Ciencias Matemáticas, UNAM Campus Morelia, Morelia, Mexico}
\email{e.roldan@im.unam.mx} 


\begin{document}

\begin{abstract}
Classic mass partition results are about dividing the plane into regions that are equal with respect to one or more measures (masses). We introduce a new concept in which the notion of partition is replaced by that of a cover. In this case we require (almost) every point in the plane to be covered the same number of times. If all elements of this cover are equal with respect to the given masses, we refer to them as equicoverings.

To construct equicoverings, we study a natural generalization of $k$-fan partitions, which we call spiral equicoverings. Like $k$-fans, these consist of wedges centered at a common point, but arranged in a way that allows overlapping. Our main result nearly characterizes all reduced positive rational numbers $p/q$ for which there exists a covering by $q$ convex wedges such that every point is covered exactly $p$ times. The proofs use results about centerpoints and combine tools from classical mass partition results, and elementary number theory.
\end{abstract}

\maketitle

\section{Introduction}

Let $\mu$ be a Borel probability measure on $\R^2$. We say that $\mu$ is a \emph{mass} if the $\mu$-measure of any line is zero.

Let $\mu_1,\dots,\mu_j$ be masses in $\R^2$.
Mass partition problems in the plane typically involve finding a partition of the plane $\Pp$ into $k$ pieces that are equal with respect to each of the masses $\mu_1,\dots,\mu_j$ \cite{RS2022}.
Such a partition $\Pp$ is called an \emph{equipartition}.
In practice, the closures of the elements of $\Pp$ are almost always polygons (possibly unbounded), so we slightly abuse the definition and assume that all pieces are closed, even if this means their boundaries may overlap.

The most well-known example of a mass partition result is the Ham Sandwich Theorem which, in the planar case, states that for any two masses there always exists an equipartition obtained by splitting the plane with a single line \cite{BZ2004}.
Another example is a theorem of Buck and Buck, which states that for a single mass, there exists an equipartition of the plane into six parts obtained by splitting the plane using three concurrent lines \cite{BB1949}.

These two examples can be generalized using a \emph{$k$-fan}, which consists of $k$ rays emanating from a single point. A $k$-fan divides the plane into regions called \emph{wedges}. If a $k$-fan determines $k$ convex wedges, it is called a \emph{convex $k$-fan}; otherwise, there is exactly one non-convex wedge. The set of wedges obtained from a $k$-fan forms a partition of the plane.

Several mass partition results involving $k$-fans are known.
For three masses, there exists an equipartition obtained from a $2$-fan \cite{BM2001}.
For two masses, there exists an equipartition obtained from a convex $3$-fan \cite{IUY2000}.
Likewise, for two masses, there exists an equipartition obtained from a $4$-fan \cite{BM2002}.
More general results can be found in \cite{BDB2007,BBDB2013,VZ2003,BBS2010}.

In this work, we no longer consider partitions of the plane, but rather covers in which almost every point is covered the same number of times. To be precise, let $\mu_1,\dots,\mu_j$ be masses in $\R^2$. A \emph{$(k,r)$-equicovering} of the plane respect to these masses is a collection $\Pp$ of $k$ sets such that:
\begin{enumerate}
	\item each set in $\Pp$ has the same $\mu_i$-measure for every $i=1,\dots,j$, and
	\item every point in $\R^2$, except those lying on the union of finitely many lines, is contained in exactly $r$ elements of $\Pp$.
\end{enumerate}

A simple way to obtain such a cover is by using $k$-fans. Assume the $k$-fan consists of rays $\ell_1,\dots,\ell_k$, ordered counter-clockwise and emanating from a point $O$. For each $i = 1,\dots,k$, we define the \emph{wedge} $V_i$ as the region obtained by sweeping the ray $\ell_i$ counter-clockwise around $O$ until it reaches the ray $\ell_{i+1}$, and the \emph{$r$-wedge} $W_i$ as the region obtained by sweeping the ray $\ell_i$ counter-clockwise until it reaches $\ell_{i+r}$ (with indices taken modulo $k$). Equivalently, $W_i = \bigcup_{j=i}^{i+r-1} V_j$.
The collection $\Pp = \{W_1,\dots,W_k\}$ then forms a cover of the plane in which almost every point is covered exactly $r$ times, except for those lying on the rays $\ell_i$. We call $\mathcal{P}$ a \emph{$(k,r)$-spiral}, and if every $W_i$ is convex, we call it a \emph{convex $(k,r)$-spiral}.

Let $\mu$ be a mass, and let $p/q < 1$ be a positive rational number expressed in reduced form (i.e., $p, q > 0$ with $p$ and $q$ coprime). We wish to determine whether there exists a convex $(q,p)$-spiral $\Pp$ such that each element of $\Pp$ has $\mu$-measure $p/q$. We call such a cover a \emph{convex $(q,p)$-spiral equicovering}, or simply a \emph{spiral equicovering} when the values of $p$ and $q$ are clear from context.

When $p$ and $q$ are coprime, as above, each wedge $V_i$ has $\mu$-measure $1/q$. Moreover, note that if we remove the convexity requirement, the problem becomes trivial (see the basic construction in Section~\ref{sec:proof}).

Our main result is the following.
\begin{theorem}\label{main}
    Assume that $p/q$ is a reduced positive rational number.
    \begin{enumerate}
        \item If $q < 2p$, then there exists a mass $\mu$ that admits no spiral equicovering.
        \item If $q < 3p - 3$, then there exists a mass $\mu$ that admits no spiral equicovering.
        \item If $q = 3p - 3$ and $p$ is odd, then every mass $\mu$ admits a spiral equicovering.
        \item If $3p - 3 < q < 3p$ and $p \geq 2$, then every mass $\mu$ admits a spiral equicovering.
        \item If $q \geq 3p$, then every mass $\mu$ admits a spiral equicovering.
    \end{enumerate}
\end{theorem}
As can be seen, there is a small gap in this theorem: the problem remains open for the case $q = 3p - 3$ when $p \geq 4$ is even. In particular, the case $p/q = 4/9$ remains unresolved.

\section{Proof of Theorem~\ref{main}}\label{sec:proof}

We divide the proof of Theorem~\ref{main} into several propositions.
The examples proving parts~(1) and~(2) are constructed in Propositions~\ref{prop:below2p} and~\ref{prop:below3p}, respectively. The cases where spiral equicoverings are guaranteed are treated in the subsequent propositions: Part~(5) is proved in Proposition~\ref{prop:firstpos}, Part~(3) is proved by Proposition~\ref{prop:3p-3}, and Part~(4) is divided into two cases, established by Propositions~\ref{prop:3p-1} and~\ref{prop:3p-2}.

\begin{proposition}\label{prop:below2p}
    Let $p/q$ be a reduced positive rational number. If $q < 2p$, then no mass $\mu$ admits a $(q,p)$-spiral equicovering.
\end{proposition}

\begin{proof}
    Assume that there exists a mass $\mu$ with a spiral equicovering $\{W_1,\dots,W_q\}$ determined by rays $\ell_1,\dots,\ell_q$. Since the $p$-wedges $W_1$ and $W_{p+1}$ satisfy
    \[
        \mu(W_1) + \mu(W_{p+1}) = \frac{p}{q} + \frac{p}{q} > 1,
    \]
    their angles must add up to more than $2\pi$. Therefore, $W_1$ and $W_{p+1}$ cannot both be convex.
\end{proof}

For the second part of Theorem~\ref{main}, we recall the following definitions and results. Given a mass $\mu$, a \emph{center-point} of $\mu$ is a point $O$ such that every closed half-plane $H$ containing $O$ satisfies $\mu(H) \geq 1/3$. It is well known that a center-point exists for any mass in the plane, and that there exist masses for which the bound $1/3$ is tight and cannot be improved \cite{Mat2013}.
        
\begin{figure}
    \centering
    \includegraphics[width=0.45\textwidth]{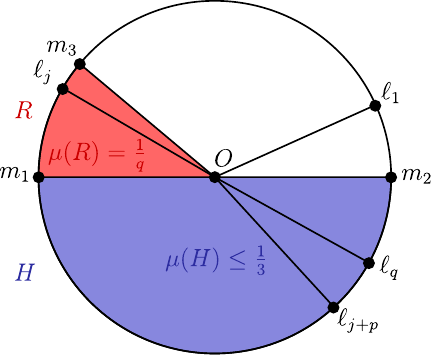}
    \caption{A mass that does not admit a spiral equicovering when $q < 3p - 3$.}
    \label{fig:pq}
\end{figure}

\begin{proposition}\label{prop:below3p}
    Let $p/q$ be a reduced positive rational number. If $q < 3p - 3$, then there exists a mass $\mu$ that admits no $(q,p)$-spiral equicovering.
\end{proposition}

\begin{proof}
Let $\mu$ be a mass with an optimal center-point; that is, for every point, there exists a closed half-plane $H$ containing it such that $\mu(H) \leq 1/3$. Suppose $q < 3p - 3$, and assume, for contradiction, that there exists a $(q,p)$-spiral equicovering $\{W_1, \dots, W_q\}$ of $\mu$, determined by a $q$-fan consisting of rays $\ell_1, \dots, \ell_q$ emanating from a point $O$, with associated wedges $V_1, \dots, V_q$. In particular, we have $\mu(W_i) = p/q$ and $\mu(V_i) = 1/q$ for each $i = 1, \dots, q$.

Let $H$ be a closed half-plane with $O$ on its boundary such that $\mu(H) \leq 1/3$.

By cyclically reordering the rays of the $q$-fan, we may assume that $\ell_1, \dots, \ell_j$ intersect $H$ only at the point $O$, and that $\ell_{j+1}, \dots, \ell_q$ are entirely contained in $H$ (see Fig.~\ref{fig:pq}).\footnote{Since our problem involves measures centered at a point $O$, the figures depict a disk in which each sector represents a corresponding wedge.} Let $m_1$ be the first counter-clockwise bounding ray of $H$ encountered starting from $\ell_j$, and let $m_2$ be the other bounding ray of $H$. Define $m_3$ as the last clockwise ray such that the counter-clockwise closed region $R$ from $m_3$ to $m_1$ has $\mu$-measure exactly $1/q$. 

The ray $\ell_j$ must lie within the region $R$, since otherwise the wedge $V_j$ would have $\mu$-measure greater than $1/q$, contradicting the assumption that $\mu(V_j) = 1/q$.

Now, since we are assuming that $W_j$ is convex, it follows that $\ell_{j+p} \subset H$. Therefore,
\[
    \mu(W_j) = \mu(W_j \setminus H) + \mu(W_j \cap H) \leq \frac{1}{q} + \frac{1}{3} < \frac{p}{q},
\]
which contradicts the assumption that $\mu(W_j) = p/q$. This completes the proof.
\end{proof}

\begin{figure}
    \centering
    \includegraphics[width=\textwidth]{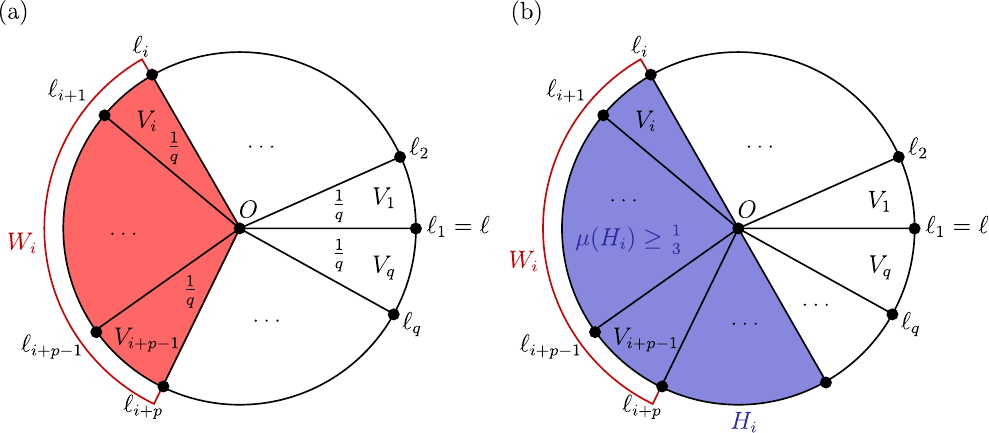}
    \caption{(a) The basic construction. (b) The construction for $q \geq 3p$.}
    \label{fig:basiccp}
\end{figure}

We now address the cases in Theorem~\ref{main} where spiral equicoverings exist. Throughout, we will repeatedly use the following \emph{basic construction}: given a mass $\mu$ and positive integers $p$ and $q$, choose a point $O$ and an initial ray $\ell$ emanating from $O$. Define $\ell_1 = \ell$, and for each $i = 1, \dots, q - 1$, recursively define $\ell_{i+1}$ as a ray such that the region swept counter-clockwise from $\ell_i$ to $\ell_{i+1}$ around $O$ has $\mu$-measure equal to $1/q$.\footnote{If some angular regions have zero $\mu$-measure, there may be more than one valid choice for $\ell_{i+1}$. Unless otherwise stated, any such choice is acceptable.}

For each $i = 1, \dots, q$, define the wedges $V_i$ and the $p$-wedges $W_i$ as before. By construction, we have $\mu(V_i) = 1/q$ and $\mu(W_i) = p/q$. See Figure~\ref{fig:basiccp}(a) for an illustration of the basic construction. In general, this procedure does not guarantee that the $p$-wedges are convex, so it is essential to choose both the center $O$ and the initial ray $\ell$ carefully.

As a warm-up, we begin with the case $q \geq 3p$.

\begin{proposition}\label{prop:firstpos}
Let $p/q$ be a reduced positive rational number. If $q \geq 3p$, then every mass $\mu$ admits a $(q,p)$-spiral equicovering.
\end{proposition}

\begin{proof}[Proof of Theorem~\ref{main}]
Assume $q \geq 3p$, and let $\mu$ be a mass. Choose a center-point $O$ of $\mu$, so that every closed half-plane containing $O$ has $\mu$-measure at least $1/3$. Select any ray $\ell$ emanating from $O$, and perform the basic construction centered at $O$ with initial ray $\ell$. We aim to show that each $p$-wedge $W_i$ is convex for $i = 1, \dots, q$.

For each $i = 1, \dots, q$, let $H_i$ be the closed half-plane obtained by sweeping $\ell_i$ counter-clockwise about $O$ through an angle of $\pi$ (see Figure~\ref{fig:basiccp}(b)). Since $O$ is a center-point, we have $\mu(H_i) \geq 1/3$. By construction,
\[
    \mu(V_i) + \cdots + \mu(V_{i+p-1}) = \frac{p}{q} \leq \frac{1}{3} \leq \mu(H_i),
\]
which implies that all of the rays $\ell_i, \dots, \ell_{i+p}$ lie within $H_i$. In particular, the terminal ray $\ell_{i+p}$ is contained in $H_i$, so the $p$-wedge $W_i$ lies entirely within a half-plane and is therefore convex.
\end{proof}

For the remaining positive results, we make use of the following well-known generalization of the theorem by Buck and Buck mentioned earlier (see~\cite{BB1949,RS2022}). An illustration is provided in Figure~\ref{fig:bb}.
\begin{lemma}\label{lem:BBG}
Let $\mu$ be a mass, and let $a, b, c \geq 0$ be real numbers satisfying $a + b + c = 1/2$. Then there exist three concurrent lines that divide $\R^2$ into six regions $R_1, \dots, R_6$, labeled in counter-clockwise order, such that:
\begin{align*}
    \mu(R_1) = \mu(R_4) &= a, \\
    \mu(R_2) = \mu(R_5) &= b, \\
    \mu(R_3) = \mu(R_6) &= c.
\end{align*}
\end{lemma}

\begin{figure}
    \centering
    \includegraphics[width=0.45\textwidth]{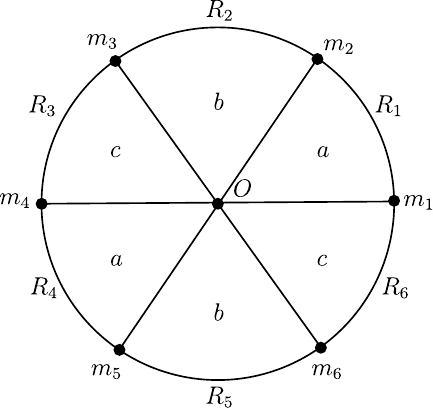}
    \caption{An illustration of Lemma~\ref{lem:BBG}: for any mass, the plane can be partitioned using three concurrent lines as described.}
    \label{fig:bb}
\end{figure}

The general strategy for the following three cases is similar. We begin by selecting appropriate values for $a$, $b$, and $c$, and then apply Lemma~\ref{lem:BBG}. When using this lemma, we denote by $O$ the point of concurrency of the three lines. For $j = 1, \ldots, 6$, let $m_j$ be the ray emanating from $O$ that separates the regions $R_{j-1}$ and $R_j$ (with indices taken modulo $6$). See Figure~\ref{fig:bb}.

For each case, we perform a basic construction that subdivides each region $R_j$ into smaller wedges of equal $\mu$-measure, ensuring that the rays $m_j$ are used as rays of the construction. 

In the case $q = 3p - 3$ with $p$ odd, a basic construction starting at $m_1$ is sufficient to obtain a spiral equicovering. The cases $q = 3p - 1$ and $q = 3p - 2$ require a basic construction that produces two orbits of $(2p)$-wedges. In the former, either orbit suffices. In the latter, we must select one of the orbits to avoid a potentially problematic obtuse angle, which could prevent convexity.

\begin{figure}
    \centering
    \includegraphics[width=\textwidth]{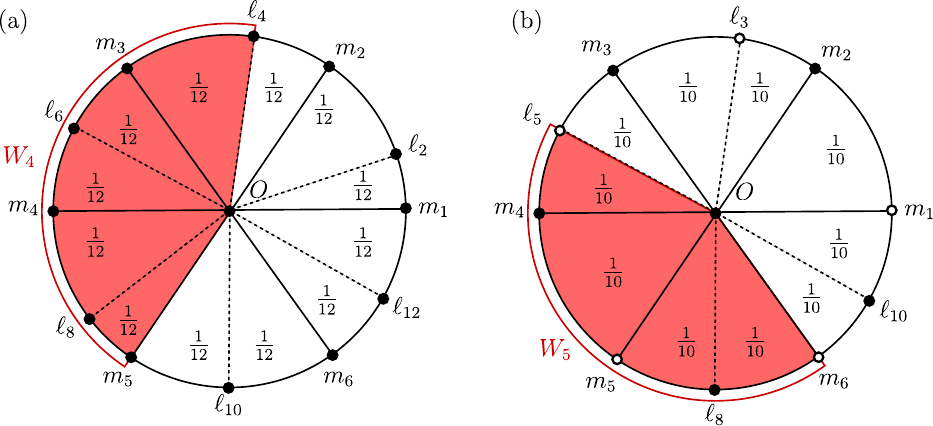}
    \caption{(a) The construction for the case $q = 3p - 3$ with $p = 5$. 
    (b) The construction for the case $q = 3p - 1$ with $p = 2$. 
    In (b), the resulting wedges are grouped into two orbits, distinguished by whether they originate at a black or white dot.}
    \label{fig:casesAB}
\end{figure}

\begin{proposition}\label{prop:3p-3}
    Let $p/q$ be a reduced positive rational number. If $q = 3p - 3$ and $p$ is odd, then every mass $\mu$ admits a $(q,p)$-spiral equicovering.    
\end{proposition}

\begin{proof}
Let $p = 2r + 1$, so that $q = 3p - 3 = 6r$, and note that $r \geq 1$. Apply Lemma~\ref{lem:BBG} with $a = b = c = 1/6$, and denote by $O$ the point of concurrency and by $m_1, \dots, m_6$ the rays between the resulting regions $R_1, \dots, R_6$, as described earlier. 

Perform a basic construction centered at $O$ with initial ray $m_1$. Since each region $R_j$ has $\mu$-measure $1/6$ and $q = 6r$, we can ensure that, for each $j = 1, \dots, 6$, the ray $\ell_{(j-1)r+1}$ coincides with $m_j$. See Figure~\ref{fig:casesAB}(a) for an example with $r = 2$.

Each $p$-wedge $W_i$ consists of $p = 2r + 1$ consecutive wedges $V_i$, and each region $R_j$ contains exactly $r$ such wedges. Therefore, every $W_i$ spans three consecutive regions $R_j$, and is thus entirely contained within a half-plane. This implies that $W_i$ is convex, completing the proof.
\end{proof}

The case $q = 3p - 1$ requires a different choice of values for $a, b, c$ in Lemma~\ref{lem:BBG}, and instead uses a basic construction based on wedges of $\mu$-measure $1/(2q)$ rather than $1/q$. With these adjustments, the overall structure of the argument remains similar to the previous case, though it requires additional care in ensuring convexity.

\begin{proposition}\label{prop:3p-1}
    Let $p/q$ be a reduced positive rational number. If $q = 3p - 1$ and $p \geq 2$, then any mass $\mu$ admits a $(q,p)$-spiral equicovering.    
\end{proposition}

\begin{proof}
Let $p = r + 1$, so that $q = 3r + 2$ with $r \geq 1$. Apply Lemma~\ref{lem:BBG} with $a = \frac{r}{6r + 4}$ and $b = c = \frac{r + 1}{6r + 4}$, and denote by $O$ the point of concurrency and by $m_1, \dots, m_6$ the rays as before.

Now, perform a basic construction centered at $O$, starting from the ray $m_1$, but use wedges of $\mu$-measure $1/2q = \frac{1}{6r + 4}$ instead of $1/q$, and define $(2p)$-wedges instead of $p$-wedges. Since the measure of each region $R_j$ is an integer multiple of $1/2q$, we can carry out the construction so that:
\[
\ell_1 = m_1,\ \ell_{r+1} = m_2,\ \ell_{2r+2} = m_3,\ \ell_{3r+3} = m_4,\ \ell_{4r+3} = m_5,\ \ell_{5r+4} = m_6.
\]
See Figure~\ref{fig:casesAB}(b) for an example with $r = 1$.

This yields $2q$ rays and $2q$ $(2p)$-wedges $W_1, \dots, W_{2q}$, each of $\mu$-measure $\frac{2p}{2q} = \frac{p}{q}$. Since $\gcd(2p, 2q) = 2$, these wedges fall into two disjoint orbits:
\[
\mathcal{O}_1 = \{W_1, W_3, \dots, W_{2q - 1}\}, \quad
\mathcal{O}_2 = \{W_2, W_4, \dots, W_{2q}\},
\]
each forming a $(q, p)$-spiral with the desired $\mu$-measure.

To complete the proof, we show that each $W_i$ is convex. Each $(2p)$-wedge $W_i$ consists of $2r + 2$ consecutive wedges $V_i$. Even in the worst-case scenario ($W_i$ beginning at $\ell_{3r+2}$ or at $\ell_{6r+4}$) the full block of $2r + 2$ wedges lies entirely within three consecutive regions $R_j$, which are themselves confined to a half-plane. Hence, each $W_i$ is convex, and either orbit $\mathcal{O}_1$ or $\mathcal{O}_2$ gives a valid $(q, p)$-spiral equicovering.
\end{proof}

In the previous case $q = 3p - 1$, the construction yielded two disjoint orbits of $(2p)$-wedges, and each orbit formed a valid $(q,p)$-spiral equicovering. In the case $q = 3p - 2$, we follow a similar construction. However, the convexity condition may fail for one of the two orbits due to the angular configuration of the wedges.
We show that although convexity may fail for one orbit, it always holds for the other. Thus, at least one of the two orbits will yield a valid convex $(q,p)$-spiral equicovering.

\begin{figure}
    \centering
        \includegraphics[width=\textwidth]{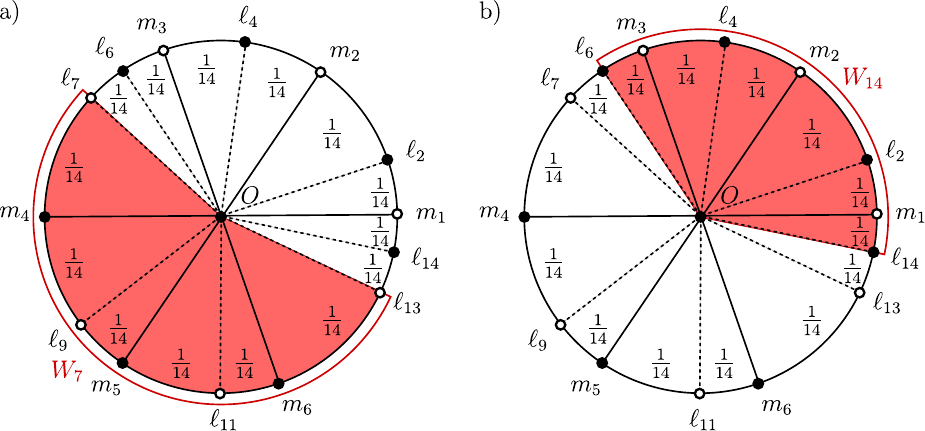}
    \caption{Both figures depict the same construction for case $q=3p-2$, with $p=3$. Figure (a) shows a possibly non-convex $W_7$ in the white orbit. Figure (b) shows that $W_{14}$, in the black orbit, is disjoint from $W_7$, so it must be convex. Each of the remaining $6$-wedges in the black orbit lies in three consecutive regions $R_j$.}\label{fig:lastcase}
\end{figure}

\begin{proposition}\label{prop:3p-2}
Let $p/q$ be a reduced positive rational number. If $q = 3p - 2$ and $p \geq 2$, then any mass $\mu$ admits a $(q,p)$-spiral equicovering.    
\end{proposition}

\begin{proof}
If $p$ is even, then $q = 3p - 2$ is also even, so $p$ and $q$ cannot be coprime. Hence, $p$ must be odd. Write $p = 2r + 1$, so that $q = 3p - 2 = 6r + 1$, with $r \geq 1$.

Use Lemma~\ref{lem:BBG} with $a = b = \frac{2r}{12r + 2}$ and $c = \frac{2r + 1}{12r + 2}$, and denote the center of concurrency by $O$ and the six separating rays by $m_1,\dots,m_6$, as above. Perform a basic construction with center $O$ and initial ray $m_1$, but use $\mu$-measure $1/2q$ instead of $1/q$, and define $(2p)$-wedges instead of $p$-wedges. Since the $\mu$-measure of each region $R_j$ is an integer multiple of $1/2q = \frac{1}{12r + 2}$, we may carry out the construction so that:
\[
\ell_1 = m_1,\ \ell_{2r+1} = m_2,\ \ell_{4r+1} = m_3,\ \ell_{6r+2} = m_4,\ \ell_{8r+2} = m_5,\ \ell_{10r+2} = m_6.
\]
See Figure~\ref{fig:lastcase} for an example with $r = 1$.

The construction yields $(2q)$-wedges $W_1,\dots,W_{2p}$ of $\mu$-measure $\frac{2p}{2q} = \frac{p}{q}$. Since $\gcd(2p, 2q) = 2$, these wedges are split into two orbits:
\begin{align*}
    \mathcal{O}_1 &= \{W_1, W_3, \dots, W_{2q-1}\}, \\
    \mathcal{O}_2 &= \{W_2, W_4, \dots, W_{2q}\},
\end{align*}
each of which forms a $(q,p)$-spiral equicovering consisting of wedges with the desired $\mu$-measure.

Each of these wedges consists of $4r+2$ consecutive wedges $V_i$ from the basic construction. In all cases, except for $W_{6r+1}$ and $W_{12r+2}$, the wedges $V_i$ lie entirely within three consecutive regions $R_j$, ensuring convexity. However, $W_{6r+1}$ and $W_{12r+2}$ may span more than three regions. However, these two wedges belong to different orbits and are disjoint, so at least one of them must have an angle of at most $\pi$, and hence is convex. In Figure~\ref{fig:lastcase}(a) the wedge $W_7$ fails, but the wedge $W_{14}$ in Figure~\ref{fig:lastcase}(b) does not. Therefore, at least one of the orbits $\mathcal{O}_1$ or $\mathcal{O}_2$ yields a valid $(q,p)$-spiral equicovering.
\end{proof}

\section{Other Equicoverings}

Previously, we introduced a general method for constructing $(q,p)$-equicoverings via spiral configurations. This technique enables the construction of such equicoverings for a wide range of values of $p$ and $q$. Moreover, in certain cases, it is possible to construct multiple distinct spiral equicoverings.

The existence of non-spiral equicoverings, however, remains poorly understood. In our search for such configurations, we found non-spiral examples of $(8,3)$-equicoverings, but could not find a general method for constructing non-spiral equicoverings.

In this section, we present the non-spiral $(8,3)$-equicoverings we discovered. Whether non-spiral equicoverings exist for other values of $(q,p)$ remains an open question.

\begin{proposition}
	Every mass $\mu$ admits a non-spiral $(8,3)$-equicovering.
\end{proposition}

\begin{proof}
Let $\ell_H$ be a line that divides the plane into two half-planes $H^+$ and $H^-$ such that $\mu(H^+) = \mu(H^-) = 1/2$. We regard $\ell_H$ as a horizontal line, with $H^+$ lying above it.

Next, we define a line $\ell_L$ with the following properties. The line $\ell_L$ subdivides $H^+$ into two regions: $H^{++}_L$ to the right of $\ell_L$ and $H^{+-}_L$ to the left. Similarly, it subdivides $H^-$ into $H^{-+}_L$ to the right of $\ell_L$ and $H^{--}_L$ to the left. These four regions satisfy
\[
\mu(H^{++}_L) = \frac{3}{8}, \quad
\mu(H^{+-}_L) = \frac{1}{8}, \quad
\mu(H^{-+}_L) = \mu(H^{--}_L) = \frac{2}{8},
\]
as shown in the upper part Figure \ref{fig:nospiral}(a).
The existence of the line $\ell_L$ follows from an asymmetric version of the Ham–Sandwich Theorem.

\begin{figure}
    \centering
    \includegraphics[width=\textwidth]{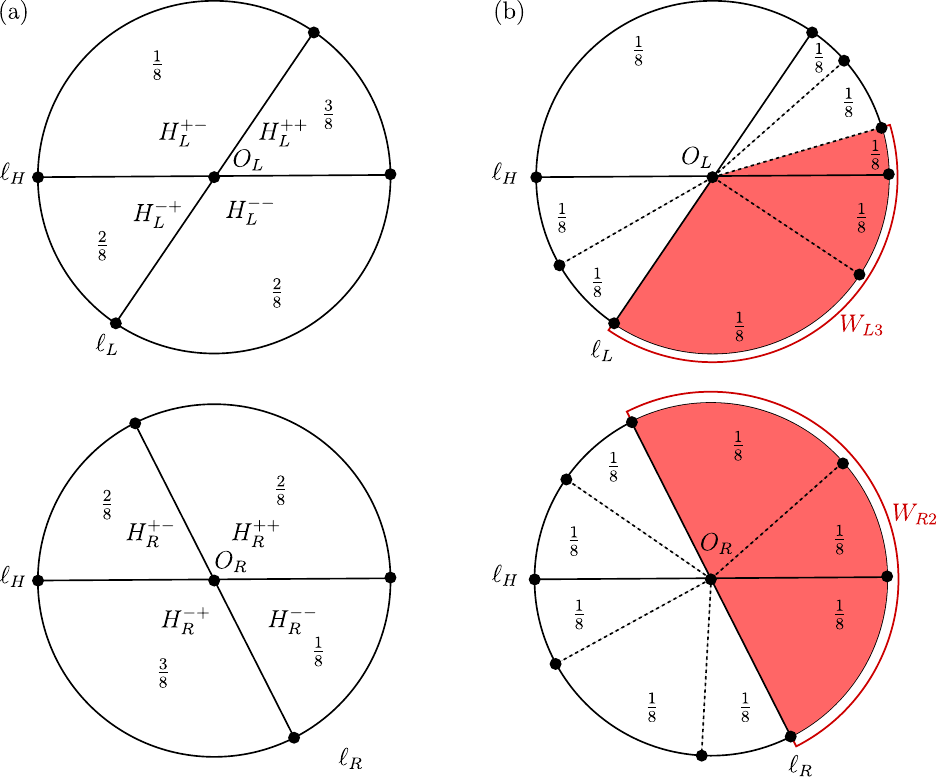}
    \caption{(a) Partitions in regions with $\mu$-measure equal to $\frac{1}{8}$, $\frac{2}{8}$, $\frac{2}{8}$ and $\frac{3}{8}$ specific measures (b) Construction of a non-spiral equicovering.}
    \label{fig:nospiral}
\end{figure}

Let $O_L$ denote the intersection point between $\ell_H$ and $\ell_L$. We now define a spiral covering centered at $O_L$ that winds around it one and a half times. Begin with a ray emanating to the right from $O_L$ and construct wedges of $\mu$-measure $1/8$. From these, consider the family containing the first four corresponding $3$-wedges $\Pp_L=\{W_{L1}, W_{L2}, W_{L3}, W_{L4}\}$. This construction is depicted in the upper part of Figure \ref{fig:nospiral}(b), where $W_{L3}$ is shown.

It is straightforward to verify that these four $3$-wedges are convex. Moreover, they cover the upper half-plane $H^+$ twice and the lower half-plane $H^-$ once.

Next, we perform an analogous construction to obtain the remaining four pieces. Define the line $\ell_R$ that divides $H^+$ and $H^-$ into the regions $H^{++}_R$, $H^{+-}_R$, $H^{-+}_R$ and $H^{--}_R$. These regions satisfy
\[
\mu(H^{++}_R) = \mu(H^{+-}_R) = \frac{2}{8}, \quad
\mu(H^{-+}_R) = \frac{1}{8}, \quad
\mu(H^{--}_R) = \frac{3}{8},
\]
as illustrated in the lower part of Figure \ref{fig:nospiral}(a).

Proceeding as before, we obtain a cover $\Pp_R=\{W_{R1}, W_{R2}, W_{R3}, W_{R4}\}$ which covers $H^+$ once and $H^-$ twice and where each element has $\mu$-measure $3/8$. This cover is shown in the lower part of Figure \ref{fig:nospiral}(b), where $W_{R2}$ is depicted.

Then the family $\Pp_L\cup\Pp_R$ forms a $(8,3)$-equicovering, as desired. Note that from the definition of $\ell_L$ and $\ell_R$ it follows that $O_L$ must be strictly to the left of $O_R$, therefore this cover cannot be a spiral equicovering.
\end{proof}

As we can see from the previous proof, there remains a degree of freedom since the line $\ell_H$ can be chosen in an arbitrary direction. This flexibility can be exploited in various ways. For instance, $\ell_H$ could bisect a second mass, or the two lines $\ell_L$ and $\ell_R$ could be required to be parallel, similar to what has been done in some equipartition results \cite{RS2014}.

It is possible that by considering general equicoverings, we could obtain positive results for a broader range of $(q,p)$ values compared to spiral equicoverings.

\section*{Acknowledgements}

This work was supported by UNAM-PAPIIT IN111923.

%

\bibliographystyle{amsalpha}
\bibliography{main}
	
\end{document}